\documentclass[letterpaper,11pt,reqno]{amsart}
\usepackage{graphicx,amsmath,amssymb,amsthm, float}

\usepackage{srcltx}
\usepackage{mathrsfs}

\usepackage[utf8]{inputenc}
\usepackage[T1]{fontenc}

\usepackage[left=1in,right=1in, top=1in, bottom=1in]{geometry}
\usepackage[hidelinks]{hyperref}

\newtheorem{X}{X}[section]

\newtheorem{lemma}[X]{Lemma}

\newtheorem{proposition}[X]{Proposition}
\newtheorem{theorem}[X]{Theorem}

\theoremstyle{definition}
\newtheorem{remark}[X]{Remark}

\newcommand{\Z}{\mathbb Z}

\newcommand{\Dstar}{\mathcal D_{k}(\phi;X)}

\newcommand{\DplusAvg}{\mathscr D_{K,h}^+(\phi)}
\newcommand{\DminusAvg}{\mathscr D_{K,h}^-(\phi)}
\newcommand{\DplusminusAvg}{\mathscr D_{K,h}^{\pm}(\phi)}

\newcommand{\eps}{\varepsilon}

\numberwithin{equation}{section}

\title[Low-lying zeros in families of holomorphic cusp forms]{Low-lying zeros in families of holomorphic cusp forms: \\ the weight aspect}

\author{Lucile Devin, Daniel Fiorilli and Anders S\"odergren}
\address{Centre de recherches mathématiques, Université de Montréal, Pavillon André-Aisenstadt, \newline \rule[0ex]{0ex}{0ex}\hspace{8pt} 2920 Chemin de la tour, Montréal, Québec, H3T 1J4, Canada}
\email{devin@crm.umontreal.ca}
\address{CNRS, Laboratoire de math\'ematiques d'Orsay, Universit\'e Paris-Sud, France }
\address{
D\'epartement de math\'ematiques et de statistique, Universit\'e d'Ottawa, 585 King Edward, Ottawa,\newline
\rule[0ex]{0ex}{0ex}\hspace{8pt} Ontario, K1N 6N5, Canada}
\email{daniel.fiorilli@math.u-psud.fr}
\address{Department of Mathematical Sciences, Chalmers University of Technology and the University \newline
\rule[0ex]{0ex}{0ex}\hspace{8pt} of Gothenburg, SE-412 96 Gothenburg, Sweden}
\email{andesod@chalmers.se}

\date{\today}
\begin{document}

\maketitle

\begin{abstract}
We study low-lying zeros of $L$-functions attached to holomorphic cusp forms of level~$1$ and large weight. In this family, the Katz--Sarnak heuristic with orthogonal symmetry type was established in the work of Iwaniec, Luo and Sarnak for test functions $\phi$ satisfying the condition supp$(\widehat \phi) \subset(-2,2)$. We refine their density result by uncovering lower-order terms that exhibit a sharp transition when the support of $\widehat \phi$ reaches the point $1$. In particular the first of these terms involves the quantity $\widehat \phi(1)$ which appeared in previous work of Fouvry--Iwaniec and Rudnick in symplectic families. Our approach involves a careful analysis of the Petersson formula and circumvents the assumption of GRH for $\text{GL}(2)$ automorphic $L$-functions. Finally, when supp$(\widehat \phi)\subset (-1,1)$ we obtain an unconditional estimate which is significantly more precise than the prediction of the $L$-functions Ratios Conjecture.  
\end{abstract}

\section{Introduction}

Katz and Sarnak~\cite{KS} conjectured that the distribution of low-lying zeros in a family $\mathcal F$ of $L$-functions is governed by a certain random matrix model $G(\mathcal F)$ called the symmetry type of $\mathcal F$. This symmetry type has been determined in many families; see for example \cite{FI,HR, ILS,M1,OS,Y}, as well as the references in \cite{SST}. Sarnak, Shin and Templier~\cite{SST} recently refined the Katz--Sarnak heuristics and introduced invariants which allow for a conjectural determination of the symmetry type. 

In the current paper we focus on the family of classical holomorphic cusp forms of level $1$ and large even weight $k$. As in \cite[Chapter 10]{ILS}, this will ease the exposition and allow for a more transparent analysis. For this family, the predictions of Katz and Sarnak were confirmed in the influential work of Iwaniec, Luo and Sarnak~\cite{ILS} for a certain class of test functions, under the assumption of the Riemann Hypothesis for Dirichlet $L$-functions and holomorphic cusp form $L$-functions. Our first goal is to relax these conditions by only assuming the Riemann Hypothesis for Dirichlet $L$-functions. Our second and main goal is to refine the Iwaniec--Luo--Sarnak density results by determining lower-order terms up to an arbitrary negative power of $\log k$. 

More precisely, we fix a basis $B_k$ of Hecke eigenforms in the space $H_k$ of holomorphic modular forms of level $1$ and even weight $k$. We normalize so that for every $$f(z)= \sum_{n=1}^{\infty}a_f(n) n^{\frac{k-1}2} e^{2\pi i n z}\in B_k,$$ the first coefficient satisfies $a_f(1)=1$. Hence, the Hecke eigenvalues of $f$ are given by $\lambda_f(n)=a_f(n)$ and for $\Re(s)>1$
the $L$-function of $f$ takes the form
$$ L(s,f)= \sum_{n=1}^{\infty} \frac{\lambda_f(n)}{n^s}. $$
This classically extends to an entire function and satisfies a functional equation relating the values at $s$ to those at $1-s$. In the sums over $f\in B_k$ to be considered in this paper, we will scale each term with the harmonic weight 
$$\omega_f:= \frac{\Gamma(k-1)}{(4\pi)^{k-1} (f,f)}, $$
where
$$ (f,f):= \int_{SL_2(\Z)  \backslash \mathbb H} y^{k-2} |f(z)|^2 dxdy. $$
Note that $ k^{-1-\eps} \ll_{\eps} \omega_f \ll_{\eps} k^{-1+\eps}$ (see \cite[Lemma 2.5]{ILS}, \cite[p.164]{HL} and \cite[Theorem 2]{Iw}), and moreover 
\begin{equation}
\Omega_k:=\sum_{f\in B_k} \omega_f = 1+O(2^{-k})
\label{equation total weight}
\end{equation}
(take $m=n=1$ in Proposition~\ref{proposition orthogonality bound}). The use of these essentially constant weights is standard (see for instance \cite[Chapter 10]{ILS}, \cite{M})\footnote{Note that the situation can be drastically different with  arithmetic weights as in \cite{KR}.}. 

For an even Schwartz test function $\phi$, we define the $1$-level density
$$ \Dstar:= \frac 1{\Omega_k}\sum_{f \in B_k} \omega_f \sum_{\gamma_f} \phi\Big( \gamma_f \frac{\log X}{2\pi} \Big). $$
Moreover, for $h$ a non-negative, not identically zero smooth weight function with compact support in $\mathbb{R}_{>0}$ we define the following averages of the $1$-level densities over families of constant sign\footnote{Note that for any $f\in H_k$, the sign of the functional equation of $L(s,f)$ is given by $(-1)^{\frac k2}$.} of the functional equation:
$$ \DplusminusAvg := \frac{1}{H^\pm(K)}\sum_{k\equiv 3\pm 1 \bmod 4}h\Big(\frac{k-1}{K}\Big)\mathcal D_k(\phi;K^2), $$
where $H^{\pm}(K) = \sum_{k\equiv 3 \pm 1 \bmod 4}h\Big(\frac{k-1}{K}\Big)$. 
The Katz--Sarnak prediction for this family (see \cite{ILS}, \cite[Conjecture 2, Section 2.7]{SST}) states that

\begin{equation}
  	\lim_{K\rightarrow \infty} \DplusminusAvg = \int_{\mathbb R} \widehat \phi \cdot \widehat W^\pm,
\label{equation KS prediction}
\end{equation}
with 
	  $$ \widehat W^+(t)=\widehat W (SO(\text{even}))(t) = \delta_0(t)+\frac{\eta(t)}2  ; \hspace{.7cm} \widehat W^-(t)=\widehat W (SO(\text{odd}))(t)= \delta_0(t)-\frac{\eta(t)}2+1,  $$
	where $\eta(t)=1$ for $|t|<1$,  $\eta(\pm 1)=\tfrac 12$ and $\eta(t)=0$ for $|t|>1$, and $\widehat \phi(\xi) := \int_{\mathbb R}  \phi(x) e^{-2\pi i \xi x}dx $.
Under the Riemann Hypothesis for Dirichlet and holomorphic cusp form $L$-functions, the estimate \eqref{equation KS prediction} was confirmed\footnote{This work has been extended to families of more general automorphic $L$-functions in~\cite{ST}.} in \cite[Theorem 1.3]{ILS} under the condition supp$(\widehat \phi)\subset (-2,2)$. 

We now state our main theorem which, in the case when the level $N=1$, refines the estimate in \cite[Theorem 1.3]{ILS} by weakening its assumptions and obtaining lower-order terms which contain a phase transition as the support of $\widehat \phi$ reaches $1$.

\begin{theorem}
\label{theorem main}
Let $\phi$ be an even Schwartz test function for which supp$(\widehat \phi) \subset (-2,2)$. Assuming the Riemann Hypothesis for Dirichlet $L$-functions, we have the estimate 
\begin{equation}
	\DplusminusAvg = \int_{\mathbb R} \widehat \phi \cdot \widehat{W}^\pm + \sum_{1 \leq j\leq J} \frac{R_{j,h}\widehat \phi^{(j-1)}(0) \pm S_{j,h}\widehat \phi^{(j-1)}(1)}{(\log K)^j}+O_{\phi,h,J}\Big( \frac 1{(\log K)^{J+1}}\Big),
\label{equation main theorem}
	  \end{equation}
where the constants $R_{j,h}$ and $S_{j,h}$ appearing in the lower-order terms only depend on the weight function $h$ (see \eqref{equation definition S}, \eqref{equation definition R 1}, and \eqref{equation definition R 2}).  
\end{theorem}
  
We deduce Theorem 1.1 from a power-saving formula for the $1$-level density (see \eqref{eq:power saving in one level density}), which we combine with an asymptotic evaluation of the resulting terms (see Theorem \ref{theorem final result}). We were able to circumvent the use of the Riemann Hypothesis for holomorphic cusp form $L$-functions by refining the estimate \cite[Corollary 2.2]{ILS} on the remainder in the Petersson trace formula (see Proposition \ref{proposition orthogonality bound}). The specific estimate we obtain is the following:
\begin{equation}
  \sum_{f\in B_k}  \omega_f \lambda_f(m)\lambda_f(n)= \delta(m,n)+ O_\eps\Big(  \frac{(m,n)^{\frac 12}(mn)^{\frac 14+\eps}}{k} +\frac{ k^{\frac 16}(m,n)^{\frac 12}}{(mn)^{\frac 14-\eps}}\Big). 
\label{equation improvement ILS}
\end{equation}
In particular, when $(m,n)=1$ this estimate is nontrivial in the range $mn\leq k^{4-\eps}$, whereas \cite[Corollary 2.2]{ILS} is nontrivial up to $mn\leq k^{\frac {10}3-\eps}$.

The terms involving $ \widehat \phi^{(j)}(1)$ in \eqref{equation main theorem} are responsible for a sharp transition at $1$ in these orthogonal families and are analogous to those obtained in symplectic families in \cite{FI,FPS2,FPS3,R,W}. Indeed, in the family of real Dirichlet characters considered in \cite{FPS2}, after applying the explicit formula and treating the resulting sums over primes by repeatedly using the Poisson summation formula, one obtains lower-order terms involving $ \widehat \phi^{(j)}(1)$. This work was inspired by the function field case considered in \cite{R}, in which, using Poisson summation, the $1$-level density is turned into an average of the trace of the Frobenius class in the hyperelliptic ensemble, from which a transition term is isolated using the explicit formula. Transition terms also surface in predictions coming from the $L$-function Ratios Conjecture \cite{FPS3,W}; in this case one needs to compute averages of ratios of local factors at infinity.
In the current situation, these terms come from a significantly different source, namely from a careful analysis of averages of Bessel functions and Kloosterman sums coming from the Petersson trace formula.  
Independently of the use of different methods, this seems to indicate that a transition in lower-order terms should exist whenever the symmetry type of a family is even or odd orthogonal, or symplectic.  
 
Interestingly, averaging over all even values of the weight $k$, we find that 
$$  \frac{1}{H(K)}\sum_{k\equiv 0 \bmod 2}h\Big(\frac{k-1}{K}\Big)\mathcal D_k(\phi;K^2) = \int_{\mathbb R}  \widehat \phi\cdot \widehat W + \sum_{1 \leq j\leq J} \frac{R_{j,h}\widehat \phi^{(j-1)}(0) }{(\log K)^j}+O_{\phi,h,J}\Big( \frac 1{(\log K)^{J+1}}\Big),  $$
where $\widehat W =\widehat W (O):=\tfrac 12+\delta_0 $ and $H(K):=H^+(K)+H^-(K)$. Hence, as expected we see that there is no transition at $1$ in this mixed signs family (see also \cite[Theorem 1.6]{M}). We should point out that for similar reasons, there is no transition in mixed sign families of holomorphic cusp forms of fixed weight and of large level~\cite{M,RR}.  

\begin{remark}
One can compute explicitly the constants $R_{j,h}$ and $S_{j,h}$ in Theorem \ref{theorem main}. In particular, the first of these are given by 
$$ R_{1,h}:=  1+ \frac{\int_0^{\infty} h\cdot \log }{\int_0^{\infty} h}-\log (4\pi) +\int_{1}^{\infty} \frac{\theta(t)-t}{t^2}dt; $$
$$S_{1,h}:=-\gamma + \frac{\int_0^\infty h \cdot \log}{ \int_0^\infty h}  - \log (4\pi) - \sum_p \frac{\log p}{p(p-1)} ,$$
where $\theta(t):=\sum_{p \leq t }\log p$.
\end{remark}

We now state our results for test functions whose Fourier transform is supported in the interval $(-1,1)$. Under this restriction our estimates are substantially more precise. Indeed, we do not need the GRH assumption, the error term is exponentially small in the weight $k$ and we do not need the average over $k$ (we set $X=k^2$ in $\mathcal D_k(\phi;X)$).

\begin{theorem}
\label{theorem small support}
Let $\phi$ be an even Schwartz test function for which supp$(\widehat \phi ) \subset (-1,1) $. Then the (unaveraged) $1$-level density satisfies the estimate
\begin{multline}
\mathcal D_k(\phi;k^2) = \frac 1{\log (k^2)} \int_{\mathbb R} \Big(  \frac{\Gamma'}{\Gamma} \Big( \frac 14+ \frac{k+ 1}4 +\frac{2\pi i t}{\log (k^2)} \Big)+ \frac{\Gamma'}{\Gamma} \Big( \frac 14+ \frac{k- 1}4 +\frac{2\pi i t}{\log (k^2)} \Big)\Big) \phi(t) dt \\+2\sum_{p} \frac 1p \widehat \phi \Big( \frac{2 \log p}{\log (k^2)} \Big) \frac{\log p}{\log (k^2)}-\widehat \phi(0)  \frac {\log \pi}{\log k}+ O(  k^{\frac 32} 2^{-k }).
\label{equation theorem small support}
\end{multline}
\end{theorem}

\begin{remark}
$ $
\begin{enumerate}

\item We emphasize that Theorem \ref{theorem small support} is unconditional. Moreover, the error term in \eqref{equation theorem small support} is exponentially small, in particular this is significantly more precise than predictions from the Ratios Conjecture \cite{CFZ,CS,M}. This comes from exponential bounds on the Bessel functions occuring in the Petersson trace formula (see \eqref{equation improved bound orthogonality}). 

\item The Katz--Sarnak main term in this case is given by 
$ \widehat \phi(0) + \frac{\phi(0)}2. $
One can extract this term from \eqref{equation theorem small support} by applying 
Lemmas \ref{lemma:gamma factors} and \ref{lemma:Seven}. 

\item An estimate for the $1$-level density $\mathcal D_k(\phi;k^2)$ was previously obtained by Miller~\cite[Lemmas 4.2 and 4.4]{M} with the error term $O_\eps (k^{\frac{\sigma}2-\frac 56+\eps})$, under the same conditions. 

\end{enumerate}

\end{remark}

Lower-order terms in the level aspect were previously studied in~\cite{M,MM,RR}.
In an upcoming paper we will refine the Iwaniec--Luo--Sarnak and Miller--Montague estimates for fixed weight and large level and isolate a transition term of the same type as in Theorem \ref{theorem main}.  

The paper is divided as follows. In Sections \ref{section 2 explicit formula} and \ref{section 3 petersson} we discuss prerequisites, establish \eqref{equation improvement ILS} and discard higher prime powers in the explicit formula. Section \ref{section 4 small support} is dedicated to the proof of Theorem \ref{theorem small support}. Finally, in Section \ref{section 5 average over weight} we apply  estimates on averages of Bessel functions to isolate a transition term, which we carefully evaluate in Section \ref{section 6 transition term}. 

\section*{Acknowledgments}
We would like to thank the Institut français de Suède, the Laboratoire de Mathématiques d'Orsay and the Centro Internazionale per la Ricerca Matematica in Trento for supporting this project and providing excellent research conditions. The first author was supported by a Postdoctoral Fellowship at the University of Ottawa. The second author was supported by an NSERC discovery grant at the University of Ottawa. The third author was supported by a grant from the Swedish Research Council (grant 2016-03759).

\section{Explicit formula}

\label{section 2 explicit formula}

We begin by recalling the explicit formula for holomorphic cusp form $L$-functions in the case where the level equals $1$. 

\begin{lemma}
\label{lemma explicit formula}
Let $\phi$ be an even Schwartz test function. We have the formula

\begin{multline}
\Dstar = -2\widehat{\phi}(0) \frac{\log\pi}{\log X}+ \frac 1{\log X} \int_{\mathbb R} \Big(  \frac{\Gamma'}{\Gamma} \Big( \frac 14+ \frac{k+ 1}4 +\frac{2\pi i t}{\log X} \Big)+ \frac{\Gamma'}{\Gamma} \Big( \frac 14+ \frac{k- 1}4 +\frac{2\pi i t}{\log X} \Big)\Big) \phi(t) dt\\
-\frac 2{\Omega_k}\sum_{f\in B_k} \omega_f \sum_{p,\nu} \frac{\alpha_f^{\nu}(p)+\beta_f^{\nu}(p) }{p^{\frac \nu 2}} \widehat \phi \Big( \frac{\nu \log p}{\log X} \Big) \frac{\log p}{\log X}. \label{Eq:explicit formula}
\end{multline}
Here, $\alpha_f(p), \beta_f(p)$ are the local coefficients of the $L$-function
$$L(s,f)= \prod_{p} \Big( 1-\frac {\alpha_f(p)}{p^s}\Big)^{-1}\Big( 1-\frac {\beta_f(p)}{p^s}\Big)^{-1} \hspace{1.356cm } (\Re(s)>1);$$
 in particular we have that $|\alpha_f(p)|=|\beta_f(p)|=1$.

\end{lemma}
\begin{proof}
For $f \in B_k$, the formula \cite[(4.11)]{ILS} reads
\begin{multline*}
\sum_{\gamma_f} \phi\Big( \gamma_f \frac{\log X}{2\pi} \Big) =\frac 1{\log X} \int_{\mathbb R} \Big( \frac{\Gamma'}{\Gamma} \Big( \frac 14+ \frac{k+ 1}4 +\frac{2\pi i t}{\log X} \Big)+ \frac{\Gamma'}{\Gamma} \Big( \frac 14+ \frac{k- 1}4 +\frac{2\pi i t}{\log X} \Big)\Big) \phi(t) dt\\
-2 \widehat{\phi}(0) \frac{\log\pi}{\log X}-2 \sum_{p,\nu} \frac{\alpha_f^{\nu}(p)+\beta_f^{\nu}(p) }{p^{\frac \nu 2}} \widehat \phi \Big( \frac{\nu \log p}{\log X} \Big) \frac{\log p}{\log X}.
\end{multline*}
Summing over $f\in B_k$ against the weight $\omega_f$ we obtain the desired formula. 
\end{proof}

We now estimate the integral involving the logarithmic derivative of the gamma function in \eqref{Eq:explicit formula}.

\begin{lemma}
\label{lemma:gamma factors}
Let $\eps>0$ and let $\phi$ be an even Schwartz test function. In the range $k\leq X^5$, we have the estimate\footnote{Using more terms in the Stirling approximation, this estimate can be refined to an asymptotic series in descending powers of $k$.\label{footnote}}
$$\frac 1{\log X} \int_{\mathbb R} \Big(  \frac{\Gamma'}{\Gamma} \Big( \frac 14+ \frac{k+ 1}4 +\frac{2\pi i t}{\log X} \Big)+ \frac{\Gamma'}{\Gamma} \Big( \frac 14+ \frac{k- 1}4 +\frac{2\pi i t}{\log X} \Big)\Big) \phi(t) dt = \widehat \phi(0) \Big( \frac{\log (k^2)-\log 16}{\log X}\Big) +O_\eps(k^{-1+\eps}).   $$

\end{lemma}
\begin{proof}
Applying Stirling's formula
$$ \frac{\Gamma'}{\Gamma} (z) =\log z + O(|z|^{-1})  $$ 
in the region $\Re(z) >0$, we see that
\begin{align*}
&\frac 1{\log X} \int_{\mathbb R} \Big(  \frac{\Gamma'}{\Gamma} \Big( \frac 14+ \frac{k+ 1}4 +\frac{2\pi i t}{\log X} \Big)+ \frac{\Gamma'}{\Gamma} \Big( \frac 14+ \frac{k- 1}4 +\frac{2\pi i t}{\log X} \Big)\Big) \phi(t) dt  \\
&=\frac 1{\log X} \int_{\mathbb R}\Big(\log \Big( \frac 14+ \frac{k+ 1}4 +\frac{2\pi i t}{\log X}\Big) +\log \Big( \frac 14+ \frac{k- 1}4 +\frac{2\pi i t}{\log X} \Big)\Big) \phi(t) dt  + O(k^{-1})\\
& =  \frac 1{\log X} \int_{-k^{\eps}}^{k^{\eps}}\Big(\log \Big( \frac 14+ \frac{k+ 1}4 +\frac{2\pi i t}{\log X}\Big) +\log \Big( \frac 14+ \frac{k- 1}4 +\frac{2\pi i t}{\log X} \Big)\Big) \phi(t) dt  + O_\eps(k^{-1})\\
& =  \frac 1{\log X} \int_{-k^{\eps}}^{k^{\eps}}\log \Big( \frac{k^2}{16} \Big) \phi(t) dt  + O_\eps(k^{-1+\eps}).
\end{align*} 
The result follows from extending the integral to $\mathbb R$.
\end{proof}

\section{The Petersson trace formula and related estimates}
\label{section 3 petersson}
In order to handle the term involving sums over prime powers in \eqref{Eq:explicit formula} we will apply the Petersson trace formula. 
For $m,n \in \mathbb Z $ and $c\in \mathbb N$ we define the Kloosterman sum
$$ S(m,n;c) := \sum_{\substack{x\bmod c \\ (x,c)=1}} e\Big( \frac{mx+n\overline{x}}{c}\Big).  $$
We will repeatedly use the classical Weil bound (see for instance \cite[Corollary 11.12]{IK})
$$S(m,n;c) \leq \tau(c) (m,n,c)^{\frac 12} c^{\frac 12}.$$

\begin{lemma}
	\label{lemma:peterssonexact}
Let $m,n,k\in \mathbb N$, with $2\mid k$. We have the exact formula	\begin{align}
  \sum_{f\in B_k}  \omega_f \lambda_f(m)\lambda_f(n)= \delta(m,n)+ 2\pi i^k  \sum_{c \geq 1} c^{-1} S(m,n,c) J_{k-1}\Big(\frac{4\pi  \sqrt{mn}}c\Big),
  \label{equation petersson exact}
	\end{align}
	where $J_{k-1}$ is the Bessel function.
\end{lemma}
\begin{proof}
See \cite{P2}, \cite[Proposition 14.5]{IK}.
\end{proof}

We recall the following bound on the Bessel function.  

\begin{lemma}
	\label{lemma:besselbounds}
Let $k\in \mathbb N$.	We have the bound
	$$ J_{k-1}(x)\ll \min\Big( \frac 1{(k-1)!} \Big(\frac{x}{2}\Big)^{k-1}, x^{-\frac 14} (|x-k+1|+k^{\frac 13})^{-\frac 14}\Big).  $$
\end{lemma}

\begin{proof}
See \cite[(2.11'); (2.11'')]{ILS}, which for the range $x\geq k^2$ follows from \cite{Wa}, specifically equations (1) p.49, (2) p.77, (6) p.78, (1), (3) p.199, (1) p.202, (4) p.250, (5) p.252, and for the remaining range follows from \cite[Theorem 2]{K}.
\end{proof}

In \cite[Chapter 2]{ILS}, this bound is shown to imply the estimate	$$  \sum_{f\in B_k}  \omega_f \lambda_f(m)\lambda_f(n)=\delta(m,n) + O(k^{-\frac 56} (mn)^{\frac 14}  \tau_3((m,n)) \log(2mn)), $$
which is non-trivial in the range $mn \leq k^{\frac{10}3 -\eps}$.
By a more careful decomposition of the the sum over $c$ in \eqref{equation petersson exact}, we establish a more precise estimate which is non-trivial in the wider range $mn \leq k^{4-\eps}$.
\begin{proposition}
\label{proposition orthogonality bound}
Let $\eps>0$, and let $m,n,k\in \mathbb N$, with $2\mid k$. We have the estimate
$$  \sum_{f\in B_k}  \omega_f \lambda_f(m)\lambda_f(n)= \delta(m,n)+ O_\eps\Big(  \frac{(m,n)^{\frac 12}(mn)^{\frac 14+\eps}}{k} +\frac{ k^{\frac 16}(m,n)^{\frac 12}}{(mn)^{\frac 14-\eps}}\Big). $$
In the range $mn \leq   k^2/(4\pi e)^2$, we have the exponentially precise estimate
\begin{equation}
 \sum_{f\in B_k}  \omega_f \lambda_f(m)\lambda_f(n)= \delta(m,n)+ O\Big( 2^{-k}(mn)^{\frac 14} \log(2mn) \prod_{p\mid (m,n)} \Big( 1+\frac 3{\sqrt p} \Big) \Big). 
 \label{equation improved bound orthogonality}
\end{equation} 
\end{proposition}
\begin{proof}
We bound the rightmost term in the statement of Lemma \ref{lemma:peterssonexact} by combining the Weil bound with Lemma \ref{lemma:besselbounds}, as follows:
\begin{align*}
&\sum_{c \geq 1} c^{-1} S(m,n,c) J_{k-1}\Big(\frac{4\pi  \sqrt{mn}}c\Big) \ll \
  \sum_{c \leq \frac{4\pi\sqrt{mn}}{k-1} -\frac{4\pi\sqrt{mn}}{k^{5/3}}}  c^{-\frac 14}\tau(c)\frac{(m,n,c)^{\frac 12}}{(mn)^{\frac 18}} \Big|\frac{4\pi \sqrt{mn}}c-k+1\Big|^{-\frac 14} \\&+  \sum_{\frac{4\pi\sqrt{mn}}{k-1} -\frac{4\pi\sqrt{mn}}{k^{5/3}} < c < \frac{4\pi\sqrt{mn}}{k-1} +\frac{4\pi\sqrt{mn}}{k^{5/3}}} c^{-\frac 14}\tau(c)\frac{(m,n,c)^{\frac 12}}{(mn)^{\frac 18}}k^{-\frac 1{12}} \\&+  \sum_{ \frac{4\pi\sqrt{mn}}{k-1} +\frac{4\pi\sqrt{mn}}{k^{5/3}}\leq c < \frac{4e\pi \sqrt{mn}}{k}} c^{-\frac 14}\tau(c)\frac{(m,n,c)^{\frac 12}}{(mn)^{\frac 18}} \Big|\frac{4\pi \sqrt{mn}}c-k+1\Big|^{-\frac 14} \\
  &+ \sum_{c \geq \frac{4e\pi \sqrt{mn}}{k}} (m,n,c)^{\frac 12} \tau(c)c^{-\frac 12}\frac 1{(k-1)!} \Big(\frac{2\pi\sqrt{mn}}{c}\Big)^{k-1}=S_1+S_2+S_3+S_4.
\end{align*}
We first bound $S_4$. To do so, note that
\begin{align*}
S_4&\ll  \frac 1{(k-1)!}  (2\pi\sqrt{mn})^{k-1}\sum_{d\mid (m,n) } d^{\frac 12}\sum_{\substack{c \geq \frac{4e\pi \sqrt{mn}}{k} \\ (c,(m,n))=d}}  \tau(c)c^{-k+\frac 12} \\
&\leq \frac 1{(k-1)!}  (2\pi\sqrt{mn})^{k-1}\sum_{d\mid (m,n) } \tau(d) d^{-k+1}\sum_{\substack{f \geq \frac{4e\pi \sqrt{mn}}{dk} \\ (f,(m,n)/d)=1}}  \tau(f)f^{-k+\frac 12}\\
&\ll  2^{-k} (mn)^{\frac 14} \log(2mn)  \prod_{p\mid (m,n)} \Big( 1+\frac 3{\sqrt p} \Big),
\end{align*}
by Stirling's approximation. Note that $S_1,S_2$ and $S_3$ are all empty whenever $mn\leq  k^2/(4\pi e)^2$ and hence \eqref{equation improved bound orthogonality} follows. 

We now assume that $mn > k^2/(4\pi e)^2$. A computation similar to the one above shows that
$$ S_2 \ll_\eps k^{-\frac 32} (mn)^{\frac 14+\eps} +(mn)^{-\frac 14+\eps} k^{\frac 16}(m,n)^{\frac 12} $$
(the second term accounts for the possibility that the sum contains only one term). 

As for $S_1$, we compute that
\begin{align*}
S_1 &\ll  \sum_{c \leq \frac{2\pi\sqrt{mn}}{k-1}}  \tau(c)\frac{(m,n,c)^{\frac 12}}{(mn)^{\frac 14}} + \sum_{\frac{2\pi\sqrt{mn}}{k-1}< c \leq \frac{4\pi\sqrt{mn}}{k-1} -\frac{4\pi\sqrt{mn}}{k^{5/3}}}  \tau(c)\frac{(m,n,c)^{\frac 12}}{(mn)^{\frac 18}} \Big|4\pi \sqrt{mn}-c(k-1)\Big|^{-\frac 14}\\
&\ll_\eps \frac{(mn)^{\frac 14} \log(2mn)}k \prod_{p\mid (m,n)} \Big(1+\frac 3{\sqrt p} \Big)+ (mn)^{-\frac 14+\eps} k^{\frac 16}(m,n)^{\frac 12}\\
&
\hspace{1cm}+ \sum_{\frac{2\pi\sqrt{mn}}{k-1}< c \leq \lfloor \frac{4\pi\sqrt{mn}}{k-1}\rfloor  -\frac{4\pi\sqrt{mn}}{k^{5/3}}}  \tau(c)\frac{(m,n,c)^{\frac 12}}{(mn)^{\frac 18}} \Big|4\pi \sqrt{mn}-c(k-1)\Big|^{-\frac 14}.
\end{align*}
Making the change of variables $b=\lfloor \frac{4\pi\sqrt{mn}}{k-1}\rfloor -c$, we see that the sum over $c$ is 
$$  \ll_\eps   (m,n)^{\frac 12}(mn)^{-\frac 18+\eps}\sum_{ \frac{4\pi\sqrt{mn}}{k^{5/3}} \leq b < \frac{2\pi\sqrt{mn}}{k-1} }  |b(k-1) |^{-\frac 14}\ll k^{-1}(mn)^{\frac 14+\eps} (m,n)^{\frac 12}. $$

In a similar way we see that $S_3\ll_\eps k^{-1}(mn)^{\frac 14+\eps} (m,n)^{\frac 12}+ (mn)^{-\frac 14+\eps} k^{\frac 16}(m,n)^{\frac 12} $, and the proof is finished.
\end{proof}

In the next lemma we apply Proposition \ref{proposition orthogonality bound} in order to discard higher prime powers in the explicit formula \eqref{Eq:explicit formula}.

\begin{lemma}\label{Lemma:prime powers}
	Assume that $k\in 2\mathbb N$,  $X \in \mathbb R_{\geq 2}$ and the even Schwartz test function $\phi$ are such that $X^{\sigma} < k^4$, where $\sigma := $sup$($supp$(\widehat \phi))$. Then we have the following estimate on the $1$-level density:
\begin{multline}
\Dstar = \widehat \phi(0) \Big( \frac{\log (k^2)-\log (16\pi^2)}{\log X}\Big) 
+2\sum_{p} \frac 1p \widehat \phi \Big( \frac{2 \log p}{\log X} \Big) \frac{\log p}{\log X}\\-2\sum_{f\in B_k} \omega_f \sum_{p} \frac{\lambda_f(p)}{p^{\frac 12}} \widehat \phi \Big( \frac{\log p}{\log X} \Big) \frac{\log p}{\log X}+O\Big(  \frac { X^{\frac{\sigma}4+\eps}}k +\frac 1{k^{\frac 13-\eps}}\Big). \label{Eq:explicit formula without prime powers}
\end{multline}
	Assuming the stronger condition $X^{\sigma}<(k/4\pi e)^2$, we have the more precise estimate
	\begin{multline}
\mathcal D_k(\phi;X) =\frac 1{\log X} \int_{\mathbb R} \Big(  \frac{\Gamma'}{\Gamma} \Big( \frac 14+ \frac{k+ 1}4 +\frac{2\pi i t}{\log X} \Big)+ \frac{\Gamma'}{\Gamma} \Big( \frac 14+ \frac{k- 1}4 +\frac{2\pi i t}{\log X} \Big)\Big) \phi(t) dt \\-2\widehat \phi(0)\frac {\log \pi}{\log X}   +2\sum_{p} \frac 1p \widehat \phi \Big( \frac{2 \log p}{\log X} \Big) \frac{\log p}{\log X}-2\sum_{f\in B_k} \omega_f \sum_{p} \frac{\lambda_f(p)}{p^{\frac 12}} \widehat \phi \Big( \frac{\log p}{\log X} \Big) \frac{\log p}{\log X}+O_\eps\Big(\frac{k^{\frac 12+\eps}}{2^{k}}\Big).
\label{eq:explicit formula without prime powers small support}
\end{multline} 
\end{lemma}

\begin{proof}
	The goal of this proof is to estimate the terms $p,\nu \geq 2$ in \eqref{Eq:explicit formula}. By the Hecke relations, the sum of those terms is equal to
	$$-\frac 2{\Omega_k}\sum_{f\in B_k} \omega_f \sum_{p,\nu \geq 2} \frac{ \lambda_f(p^{\nu})-\lambda_f(p^{\nu-2}) }{p^{\frac \nu 2}} \widehat \phi \Big( \frac{\nu \log p}{\log X} \Big) \frac{\log p}{\log X}.
	 $$ 
	 From Proposition \ref{proposition orthogonality bound}  and \eqref{equation total weight}, we see that
	
	\begin{align*}
	-&\frac 2{\Omega_k}\sum_{f\in B_k} \omega_f \sum_{p,\nu \geq 2} \frac{ \lambda_f(p^{\nu})}{p^{\frac \nu 2}} \widehat \phi \Big( \frac{\nu \log p}{\log X} \Big) \frac{\log p}{\log X} \\
	& \ll_\eps 2^{-k}\sum_{ \substack{p,\nu \geq 2 \\ p \leq \min(X^{\sigma / \nu  },( k/4\pi e)^{2/\nu}) }}    p^{-\frac \nu 4+\eps}      +   \sum_{ \substack{p,\nu \geq 2 \\ \min(X^{\sigma / \nu  },( k/4\pi e)^{2/\nu}) < p \leq X^{\sigma/\nu}}} p^{-\frac \nu 2+\eps}(k^{-1} p^{\frac \nu 4} +k^{\frac 16} p^{-\frac \nu 4}  )  \\
	&\ll_\eps k^{\frac 12+\eps} 2^{-k}+ I_{[X^{\sigma} > (k/4\pi e)^2]} (  k^{-1}  X^{\frac \sigma 4+\eps} +k^{-\frac 13+\eps}).
	\end{align*}
	 Similarly,
		\begin{align*}
	&\frac 2{\Omega_k}\sum_{f\in B_k} \omega_f \sum_{\substack{p\\ \nu \geq 3}} \frac{ \lambda_f(p^{\nu-2})}{p^{\frac \nu 2}} \widehat \phi \Big( \frac{\nu \log p}{\log X} \Big) \frac{\log p}{\log X}  \ll   2^{-k}+I_{[X^{\sigma} > (k/4\pi e)^2]}k^{-\frac 13+\eps}.  
	\end{align*}
	 The only terms left are 
	 		\begin{align*}
	 &\frac 2{\Omega_k}\sum_{f\in B_k} \omega_f \sum_{\substack{p}} \frac{ \lambda_f(1)}{p} \widehat \phi \Big( \frac{2 \log p}{\log X} \Big) \frac{\log p}{\log X}  =2\sum_{\substack{p}} \frac{ 1}{p} \widehat \phi \Big( \frac{2 \log p}{\log X} \Big) \frac{\log p}{\log X}.  
	 \end{align*}
	 We conclude the proof by applying Lemmas \ref{lemma explicit formula} and \ref{lemma:gamma factors}, and \eqref{equation total weight}.
\end{proof}

\section{$1$-level density: Unconditional results}
\label{section 4 small support}
In this section we evaluate the $1$-level density $\Dstar$ for test functions satisfying sup(supp($\phi$)) < 1, unconditionally. We begin by asymptotically evaluating the second term on the right-hand side of \eqref{Eq:explicit formula without prime powers}.

\begin{lemma}
\label{lemma:Seven}
Let $\phi$ be an even Schwartz test function. For any fixed $J\geq 1$, we have the estimate
$$
2\sum_{p} \frac 1p \widehat \phi \Big( 2\frac{ \log p}{\log X} \Big) \frac{\log p}{\log X} = \frac{\phi(0)}2 
 + \sum_{1 \leq j\leq J} \frac{c_j\widehat \phi^{(j-1)}(0)}{(\log X)^j} +O_J\Big(  \frac 1{(\log X)^{J+1}} \Big),
$$
where 
$$ c_1:=2 \int_1^{\infty} \frac{\theta(t)-t}{t^2}dt+2$$
and for $j\geq 2$,
$$ c_j:= \frac{2^j }{(j-2)!}  \int_1^{\infty}(\log t)^{j-2} \Big( \frac{\log t}{j-1}- 1 \Big) \frac{\theta(t)-t}{t^2}dt. $$
\end{lemma}

\begin{proof}
Performing summation by parts, we reach the following identity:
\begin{align*}
\frac{2}{\log X}\sum_{p} \frac {\log p}p \widehat \phi \Big( 2 \frac{ \log p}{\log X} \Big) &= \frac{\phi(0) }{2} + \frac{2\widehat \phi(0)}{\log X} - \frac 2{\log X}  \int_1^{\infty} \Big( 2\frac{\widehat \phi'\big(2 \frac{\log t}{\log X}\big)}{\log X}- \widehat \phi \Big(2 \frac{\log t}{\log X}\Big)\Big)\frac{\theta(t)-t}{t^2}dt.
\end{align*}    
By the prime number theorem in the form $\theta(t)-t \ll t\exp(-2c\sqrt {\log t})$, we see that for any $0<\xi<1$,
$$\int_{X^{\xi/2}}^{\infty} \Big( 2 \frac{\widehat \phi'\big(2 \frac{\log t}{\log X}\big)}{\log X}- \widehat \phi \Big(2 \frac{\log t}{\log X}\Big)\Big)\frac{\theta(t)-t}{t^2}dt \ll \exp(-c \sqrt{\xi \log X}). $$
Moreover, taking Taylor series and applying the prime number theorem, we see that 

\begin{align*}
-&\frac 2{\log X}\int_{1}^{X^{\frac{\xi}2} } \Big( 2 \frac{\widehat \phi'(2\frac{\log t}{\log X})}{\log X}- \widehat \phi \Big(2\frac{\log t}{\log X}\Big)\Big)\frac{\theta(t)-t}{t^2}dt \\
&= 2\sum_{0\leq j\leq J} \frac 1{j!} \Big(\widehat\phi^{(j)}(0) -2 \frac{\widehat\phi^{(j+1)}(0)}{\log X} \Big)\int_{1}^{X^{\frac{\xi}2} } \frac {(2\log t)^{j}}{(\log X)^{j+1}}\frac{\theta(t)-t}{t^2}dt+O_J\Big( \int_{1}^{X^{\frac{\xi}2} }\frac {(2\log t)^{J+1}}{(\log X)^{J+2}}\frac{\theta(t)-t}{t^2}dt\Big)  \\
&= \sum_{1\leq j\leq J+1} \frac{c_j\widehat\phi^{(j-1)}(0)}{(\log X)^j} +O_J\Big(  \frac 1{(\log X)^{J+2}}+\exp(-c \sqrt{\xi \log X})\Big).
\end{align*}
The result follows from selecting $\xi = (\log X)^{-1+\delta} $ for some $\delta>0$.
\end{proof}

We now set $X=k^2$ and prove Theorem \ref{theorem small support}.

\begin{proof}[Proof of Theorem \ref{theorem small support}]
We apply Proposition \ref{proposition orthogonality bound} and obtain that the second prime sum in \eqref{eq:explicit formula without prime powers small support} satisfies the bound
$$  2\sum_{f\in B_k} \omega_f \sum_{p} \frac{\lambda_f(p)}{p^{\frac 12}} \widehat \phi \Big( \frac{\log p}{\log (k^2)} \Big) \frac{\log p}{\log (k^2)} \ll k^{\frac{3\sigma}2}2^{-k}.$$
The proof follows.
\end{proof}

\section{$1$-level density averaged over the weight: Extended support}
\label{section 5 average over weight}

In this section we study the quantities $\DplusAvg$ and $\DminusAvg$, that is we average the $1$-level density
$\mathcal D_{k}(\phi;K^2)$ over $k \asymp K$ against the weight $h(\tfrac {k-1}K)$.

\begin{lemma}[{\cite[Lemma 5.8]{Iw2}, \cite[Corollary 8.2]{ILS}}]
	\label{lemma:ILS bessel sum}
	For $h$ a non-negative, smooth function with compact support in $\mathbb{R}_{>0}$ and for any $K\geq 2$, we have the estimates
	$$ 2\sum_{k\equiv 0 \bmod 2}h\Big(\frac{k-1}{K}\Big)J_{k-1}(x) = h\Big(\frac xK\Big) + O\Big(\frac x{K^3}\Big);$$
	$$ 2\sum_{k\equiv 0 \bmod 2}i^k h\Big(\frac{k-1}{K}\Big)J_{k-1}(x) = \frac{K}{\sqrt{x}}\Im\Big( \overline{\zeta_8}e^{ix}\hbar\Big(\frac{K^2}{2x}\Big) \Big) + O\Big(\frac x{K^4}\Big),$$
	where $\zeta_8 = e^{2\pi i/8}$ and $\hbar(x)=\int_0^\infty \frac{h(\sqrt{u})}{\sqrt{2\pi u}}e^{ixu}du$.
\end{lemma}

In the next lemma we estimate the total weight $H^{\pm}(K)$ and a related sum. 
\begin{lemma}
\label{lemma:average h log}
For $h$ a non-negative, smooth function with compact support in $\mathbb{R}_{>0}$ and for any $K,N\geq 2$, we have the estimates
\begin{equation*}
H^{\pm}(K) = \frac{K \int_{\mathbb R^+} h}4+O_N(K^{-N});
\end{equation*}  

\begin{equation*}
 4\sum_{k\equiv 3\pm 1 \bmod 4} h\Big(\frac{k-1}{K}\Big) \log k  = K \log K \int_{\mathbb R^+} h+ K\int_{\mathbb R^+} h\cdot \log + \sum_{\ell = 1}^{N} \frac{(-1)^{\ell+1}}{\ell K^{\ell-1}} \int_{\mathbb R^+} t^{-\ell }h(t) dt +   O_N(K^{-N}).
\end{equation*}
\end{lemma}

\begin{proof}
More generally, we will show that for any $a\bmod 4, $
\begin{equation}
\sum_{k\equiv a \bmod 4}h\Big(\frac{k}{K}\Big) = \frac{K \int_{\mathbb R^+} h}4+O_N(K^{-N});
\label{equation1averagehloggeneral}
\end{equation}  
\begin{equation}
  4\sum_{k\equiv a \bmod 4}h\Big(\frac{k}{K}\Big) \log (k+1) = K \log K \int_{\mathbb R^+} h+ K\int_{\mathbb R^+} h\cdot \log + \sum_{\ell = 1}^{N} \frac{(-1)^{\ell+1}}{\ell K^{\ell-1}} \int_{\mathbb R^+} t^{-\ell }h(t) dt +   O_N(K^{-N}). 
\label{equation2averagehloggeneral}
\end{equation}
Now, for any $b \bmod 4,$ Poisson summation gives
$$ \sum_{k\in \mathbb Z} e\Big(\frac{bk}{4}  \Big) h\Big(\frac{k}{K}\Big)=K\sum_{k\in \mathbb Z} \widehat h\Big(\Big(k-\frac b4\Big)K\Big)= K \widehat h(0) \delta_{b=0}+O_N(K^{-N}). $$
The estimate \eqref{equation1averagehloggeneral} follows by orthogonality of additive characters. Similarly, we see that
\begin{equation}
\sum_{k\equiv a \bmod 4}h\Big(\frac{k}{K}\Big) \log(k+1) = \frac{ \int_{\mathbb R^+} h\big(\frac tK\big) \log(t+1) dt}4+O_N(K^{-N}).
\label{equation3averagehloggeneral}
\end{equation}  
Indeed, integration by parts shows that 
$$  \int_{\mathbb R^+} h\Big(\frac tK\Big) \log(t+1) e( - \xi t ) dt \ll_N \frac {\log K}{(|\xi|K)^N}. $$
Finally, the integral on the right-hand side of \eqref{equation3averagehloggeneral} equals
\begin{align*}
K \int_{\mathbb R^+} h(t) \log(Kt+1) dt = K \log K \int_{\mathbb R^+} h+ K\int_{\mathbb R^+} h\cdot \log + \sum_{\ell = 1}^{N} \frac{(-1)^{\ell+1}}{\ell K^{\ell-1}} \int_{\mathbb R^+} t^{-\ell }h(t) dt +   O_N(K^{-N}),
\end{align*}
and \eqref{equation2averagehloggeneral} follows.
\end{proof}

In the next lemma we estimate the average of \eqref{Eq:explicit formula without prime powers} over $k$. In order to do so, we will apply Lemmas \ref{lemma:ILS bessel sum} and \ref{lemma:average h log}.
\begin{lemma}
\label{lemma:oneleveldeninterms of kloosterman}
Let $\phi$ be an even Schwartz test function and let $h$ be a non-negative, smooth function with compact support in $\mathbb{R}_{>0}$. Under the condition $\sigma=$sup$($supp$(\widehat \phi)) <2$ and for $K\geq 2$, we have the estimate 
	\begin{multline*}
	\DplusminusAvg = \widehat \phi(0) \bigg(1 +\frac{\frac{ \int_{\mathbb R^+} h \cdot \log}{\int_{\mathbb R^+} h }-\log (4\pi)}{\log K} \bigg) +2\sum_{p} \frac 1p \widehat \phi \Big( \frac{2 \log p}{\log (K^2)} \Big) \frac{\log p}{\log (K^2)}\\ \mp \frac {\pi}{\log (K^2)H^{\pm}(K)} \sum_{p} \frac{ \log p}{p^{\frac 12}} \widehat \phi \Big( \frac{ \log p}{\log (K^2)} \Big)\sum_{c=1}^{\infty} \frac{S(p,1;c)}c  h\Big( \frac{4\pi \sqrt{p}}{cK}\Big)+ O_\eps(K^{\frac{\sigma}2-1+\eps}+K^{-\frac 13+\eps}).
	  \end{multline*}
\end{lemma}

\begin{proof}
	From combining Lemmas ~\ref{Lemma:prime powers} and \ref{lemma:average h log}, we have that
	\begin{multline*}
	\DplusminusAvg = \widehat \phi(0) \Big(1 +\frac{\frac{ \int_{\mathbb R^+} h \cdot \log}{\int_{\mathbb R^+} h }-\log (4\pi)}{\log K} \Big) +2\sum_{p} \frac 1p \widehat \phi \Big( \frac{2 \log p}{\log (K^2)} \Big) \frac{\log p}{\log (K^2)} \\ - \frac{2}{H^{\pm}(K)}\sum_{k\equiv 3\pm 1 \bmod 4}h\Big(\frac{k-1}{K}\Big)\sum_{f\in B_k} \omega_f \sum_{p} \frac{\lambda_f(p)}{p^{\frac 12}} \widehat \phi \Big( \frac{\log p}{\log (K^2)} \Big) \frac{\log p}{\log (K^2)} + O_\eps(K^{\frac{\sigma}2-1+\eps}+K^{-\frac 13+\eps}).
	\end{multline*}
	
	By the Petersson trace formula (Lemma \ref{lemma:peterssonexact}), the third term is equal to
	\begin{align}
	&-\frac{2 \pi}{H^{\pm}(K)}\sum_{k\equiv 0 \bmod 2}( i^k\pm 1)h\Big(\frac{k-1}{K}\Big) \sum_{p} \frac{1}{p^{\frac 12}} \widehat \phi \Big( \frac{\log p}{\log (K^2)} \Big) \frac{\log p}{\log (K^2)}  \sum_{c \geq 1} c^{-1} S(p,1,c) J_{k-1}\Big(\frac{4\pi  \sqrt{p}}c\Big).
	\label{equation after pettersson third term}
	\end{align}
	
	 Applying Lemma \ref{lemma:ILS bessel sum}, we see that  
	 \begin{multline}
	 - \frac{2}{H^\pm(K)}\sum_{k\equiv 0 \bmod 2}i^kh\Big(\frac{k-1}{K}\Big)J_{k-1}\Big(\frac{4\pi \sqrt{p}}c\Big) =  \frac{-Kc^{\frac 12}}{H^\pm(K) 2\pi^{\frac 12}  p^{\frac 14}}\Im\Big( \overline{\zeta_8}e^{\frac{i4\pi  \sqrt{p}}c}\hbar\Big(\frac{K^2c }{8\pi  \sqrt{p}}\Big) \Big) + O\Big( \frac{\sqrt{p}}{cK^5}\Big).
	 \label{equation after ptersson third term evaluation}
	 	 \end{multline}
	 	 Since $p\leq K^{4-\eps}$, we see by the rapid decay of $\hbar$ that for any $A\geq 1$, the first term in this expression is 
$$ \ll_A  \frac{c^{\frac 12}}{p^{\frac 14}} \Big(\frac{K^2c}{  \sqrt{p}}\Big)^{-A} ,$$
and hence the contribution of this term to \eqref{equation after pettersson third term} is
$$ \ll_A K^{A(\sigma-2)+\frac \sigma 2 }.  $$
As for the sum of the error terms in \eqref{equation after ptersson third term evaluation}, the contribution is $\ll K^{2\sigma-5}$, which is an admissible error term.
	 Moreover, applying Lemma \ref{lemma:ILS bessel sum} once more,   
	 \begin{align*}
	 - \frac{2}{H^\pm(K)}\sum_{k\equiv 0 \bmod 2}h\Big(\frac{k-1}{K}\Big)J_{k-1}\Big(\frac{4\pi \sqrt{p}} c\Big) =  -\frac{1}{H^\pm(K)}h\Big( \frac{4\pi \sqrt{p}}{cK}\Big) + O\Big( \frac{\sqrt{p}}{cK^4}\Big),
	 	 \end{align*}
	resulting in a main term as well as the admissible error term $O(K^{2\sigma-4})$.
\end{proof}

We now end this section by evaluating the second sum over primes in Lemma \ref{lemma:oneleveldeninterms of kloosterman}, under GRH for Dirichlet $L$-functions. This term will be responsible for the phase transition at $1$, and will be investigated more closely in Section \ref{section 6 transition term}.

\begin{lemma} \label{lemma:main term integral}
	Let $\phi$ be an even Schwartz test function, let $h$ be a non-negative, smooth function with compact support in $\mathbb{R}_{>0}$, and assume the Riemann Hypothesis for Dirichlet $L$-functions. Then for any $K\geq 2$ we have the estimate
	\begin{multline} 
\label{equation lemma GRH}	
	\sum_{c\geq 1} \frac 1c \sum_{p} \frac{ \log p}{p^{1/2}} \widehat \phi \Big( \frac{ \log p}{\log (K^2)} \Big) S(p,1;c)  h\Big( \frac{4\pi \sqrt{p}}{cK}\Big) =   \log (K^2) \int_{0}^{\sigma} K^{u} \widehat \phi(u)\sum_{c\geq 1}\frac {\mu^2(c)}{c\varphi(c)} h\Big( \frac{4\pi K^{u-1}}{c}\Big) du \\+ O(K^{\sigma-1} (\log K)^3),
	\end{multline}
	where $\varphi$ is Euler's totient function.
\end{lemma}

\begin{proof}
If $\sigma < 1$, then for large enough $K$ the left-hand side of \eqref{equation lemma GRH} is identically zero. We may thus assume that $\sigma\geq 1$. The sum over $p$ equals
	\begin{align*}
	\int_{0}^{\infty}   \frac{ 1}{t^{\frac 12}} \widehat \phi \Big( \frac{ \log t}{\log (K^2)} \Big) h\Big( \frac{4\pi \sqrt{t}}{cK}\Big) dT(t),
	\end{align*}
	where, by \cite[Lemma 6.1]{ILS},
	\begin{align*}
	T(t):= \sum_{p\leq t}  S(p,1;c)\log p = t \frac {\mu^2(c)}{\varphi(c)} +O(\varphi(c) t^{\frac 12}  (\log(ct))^2).
	\end{align*}

Note that our restriction on the support of $h$ implies that $c\asymp \sqrt t/K$, and hence the restriction on the support of $\widehat \phi$ implies that for squarefree values of $c$ and for $t\leq K^{4-\epsilon}$,  the main term in this estimate is always larger than the error term. After a straightforward calculation, we obtain that the left-hand side of \eqref{equation lemma GRH} equals 
	$$ \sum_{c\geq 1} \frac{\mu^2(c)}{c\varphi(c)}\int_{1}^{\infty}   \frac{ 1}{t^{\frac 12}} \widehat \phi \Big( \frac{ \log t}{\log (K^2)} \Big) h\Big( \frac{4\pi \sqrt{t}}{cK}\Big) dt + O(K^{\sigma-1} (\log K)^3).  $$
	The proof follows. 	
\end{proof}

\section{Evaluation of the transition term}
\label{section 6 transition term} 

The goal of this section is to evaluate the integral in Lemma \ref{lemma:main term integral}. This will be done using different techniques depending on the range of the variable $u$. To this end, for $a,b \in \mathbb R_{\geq 0}$ we define
 $$I_{a,b}:=\frac {\pi}{H^{\pm}(K)} \int_{a}^{b} K^{u} \widehat \phi(u)\sum_{c\geq 1}\frac {\mu^2(c)}{c\varphi(c)} h\Big( \frac{4\pi K^{u-1}}{c}\Big) du.
$$
Notice that the inner sum is long only when $u$ is larger and far away from $1$.
By Lemmas \ref{lemma:oneleveldeninterms of kloosterman} and \ref{lemma:main term integral}, we see that when $\sigma =$sup$($supp$(\widehat{\phi})) <2$ and under the assumption of GRH for Dirichlet $L$-functions,
\begin{equation}
	\DplusminusAvg =  \widehat \phi(0) \bigg(1 +\frac{\frac{ \int_{\mathbb R^+} h \cdot \log}{\int_{\mathbb R^+} h }-\log (4\pi)}{\log K} \bigg)+2\sum_{p} \frac 1p \widehat \phi \Big( \frac{2 \log p}{\log (K^2)} \Big) \frac{\log p}{\log (K^2)} \mp I_{0,\infty} + O_{\epsilon}(K^{\frac{\sigma}2-1+\eps}+K^{-\frac 13+\eps}).
\label{eq:power saving in one level density}	  \end{equation}

We now move on to evaluating the integral $I_{0,\infty}$. We let $\delta_K$ be a positive parameter which satisfies $\delta_K \gg_h 1/\log K$.
Recall that $h$ is supported in $\mathbb R_+$, and hence for $K$ large enough the integrand in $I_{0,\infty}$ is zero in the interval $[0,1-\delta_K)$. Hence,
$$I_{0,\infty} = I_{1-\delta_K,\sigma},$$
where, as before, $\sigma=$sup$($supp$(\widehat \phi))$.

\begin{lemma}
We have the unconditional estimate\footnote{The error term here can be improved to $O(x^{\frac 12} \exp(-c(\log x)^{\frac 35} (\log\log x)^{-\frac 13}))$ by replacing \eqref{equation estimate S_1 squarefree} with the stronger estimate obtained from combining \cite[Exercise 19, \S 6.2.1]{MV} with the Korobov-Vinogradov zero-free region.} 
$$\sum_{c\leq x}\frac {c\mu^2(c)}{\varphi(c)}  = x+ O(x^{\frac 12} ).$$
\label{lemma theta was found}
\end{lemma}

\begin{proof}
We first establish the following estimate, for squarefree values of $d$:
\begin{equation}
S_d(x):=\sum_{\substack{m\leq \frac xd \\ (m,d)=1}}\frac {\mu^2(m)}{m} = C_1(d)\log x +C_2(d) +O\Big(x^{-\frac 12} d^{\frac 12 }\prod_{p\mid d} \big( 1-p^{-\frac 12}\big)^{-1}\Big),
\label{equation old lemma}
\end{equation}
 where 
$$ C_1(d) := \frac 1{\zeta(2)}\prod_{p\mid d} \Big( 1-\frac 1{p+1}\Big); \hspace{1cm} C_2(d) :=C_1(d) \Big( \gamma-2\frac{\zeta'}{\zeta}(2)-\sum_{p\mid d} \frac{p\log p}{p+1} \Big).   $$
To do so, note that
$$ S_d(x) = \sum_{\ell_1 \mid d}\frac{\lambda(\ell_1)}{\ell_1} S_{\ell_1} \Big( \frac x {d}\Big). $$
Applying this equality iteratively, we reach the identity
\begin{equation}
 S_d(x) = \sum_{\substack{\ell_1  \mid d \\ \ell_2 \mid \ell_1 \\ \vdots \\ \ell_k \mid \ell_{k-1}}}\frac{\lambda(\ell_1)\cdots \lambda(\ell_k)}{\ell_1\cdots \ell_k} S_{\ell_k} \Big( \frac x {d\ell_1 \cdots \ell_{k-1}}\Big) = \sum_{ \substack{\ell \mid d^{\infty} }  } \frac{\lambda(\ell)}{\ell} S_1\Big(\frac x{d\ell}\Big). 
\label{equation identity with ells}
\end{equation} 
A summation by parts combined with \cite[Theorem 8.25]{NZM} yields that\footnote{The precise value of the constant is deduced from writing
$S_1(x) = \frac 1{2\pi i} \int_{(1)} \frac{\zeta(s+1)}{\zeta(2s+2)} \frac{x^s }{s} ds$ and shifting the contour of integration to the left.} 
\begin{equation}
S_1(x) = \frac{1}{\zeta(2)} \Big(\log x+\gamma-2\frac{\zeta'}{\zeta}(2) \Big)+ O(x^{-\frac 12} ).
\label{equation estimate S_1 squarefree}
\end{equation} 
Inserting this estimate into \eqref{equation identity with ells}, we are left with an error term which is 
$$ \ll d^{\frac 12} x^{-\frac 12} \sum_{ \substack{\ell \mid d^{\infty}  }  } \frac{1}{\ell^{\frac 12}} \leq d^{\frac 12}x^{-\frac 12} \prod_{p\mid d} \sum_{\alpha\geq 0} \frac{1}{p^{\frac {\alpha}2}} = d^{\frac 12}x^{-\frac 12} \prod_{p\mid d} \big( 1-p^{-\frac 12}\big)^{-1}, $$
and \eqref{equation old lemma} follows. The claimed estimate then follows from the convolution identity
\begin{align*}
\sum_{c\leq x}\frac {\mu^2(c)}{\varphi(c)} =\sum_{c\leq x}\frac {\mu^2(c)}{c} \sum_{d\mid c } \frac{\mu^2(d)}{\varphi(d)} = \sum_{d\leq x}\frac{\mu^2(d)}{d\varphi(d)}  \sum_{\substack{m\leq \frac xd \\ (m,d)=1}}\frac {\mu^2(m)}{m}
\end{align*}
and a straightforward summation by parts.
\end{proof}

We now evaluate the part of the integral $I_{0,\infty}$ for which $u$ is slightly larger than $1$. In this range, the sum over $c$ is fairly long and we can effectively apply Lemma \ref{lemma theta was found}.
\begin{lemma}[Range $u>1+\delta_K$]
\label{lemma range u large}
Let $\phi$ be an even Schwartz test function and let $K\geq 2$. We have that	$$ I_{1+\delta_K,\infty} =    \int_{1+\delta_K}^{\infty} \widehat{\phi}+O(K^{-\frac{\delta_K}2}). $$
\end{lemma}
\begin{proof}
By Lemma \ref{lemma theta was found}, we have that
$$  S(y):=\sum_{c\leq y}\frac {c\mu^2(c)}{\varphi(c)}  = y+ O(y^{\frac 12}),$$
hence, for $u>1+\delta_K$,
\begin{align*}
\sum_{c \geq 1}\frac {\mu^2(c)}{c\varphi(c)} h\Big( \frac{4\pi K^{u-1}}{c}\Big) &= -\int_0^\infty S(y) \Big(\frac{1}{y^2}h\Big( \frac{4\pi  K^{u-1}}{y}\Big) \Big)' dy \\
&=\int_0^\infty \frac{1}{y^2}h\Big( \frac{4\pi  K^{u-1}}{y}\Big) dy\\
&\hspace{1cm}+ O\Big(\int_0^\infty \Big(\frac 1 {y^3} h \Big( \frac{4\pi  K^{u-1}}{y} \Big)+ \frac{  K^{u-1}}{y^4} h' \Big( \frac{4\pi  K^{u-1}}{y} \Big) \Big) y^{\frac 12} dy  \Big) \\
&= \frac{1}{4\pi  K^{u-1}}\int_0^\infty h +O(K^{-\frac 32 (u-1)}).
\end{align*}
The desired estimate follows by integrating over $u$ against $ K^u \widehat \phi(u)$ and applying Lemma \ref{lemma:average h log}.
\end{proof}

We now evaluate the part of the integral $I_{0,\infty}$ in which $u$ is close to $1$. In this range we can expand $\widehat \phi (u)$ into Taylor series around $u=1$ and recover the transition terms $\widehat \phi^{(j)}(1)$ (see Lemma \ref{lemma taylor series}). The resulting integrals are evaluated in Lemma \ref{lemma:MellinRangeclose to 1} by applying the inverse Mellin transform, truncating the resulting integrals and shifting the contours of integration.  
\begin{lemma}
\label{lemma bound incomplete gamma}
Whenever $x\in \mathbb R_{\geq 0}, j\in \mathbb N, \Re(s)>0 $ and $x|s|\geq 2$,
$$ \int_x^{\infty} u^j e^{-us} du \ll \frac{j! e^{-\Re(s)x }x^j}{|s|}, $$
where the implied constant is absolute.
\end{lemma}
\begin{proof}
Applying integration by parts, we reach the exact formula
$$\int_x^{\infty} u^j e^{-us} du = \frac{e^{-xs} x^j}{s}  \sum_{0\leq \ell \leq j} \frac{j!x^{-\ell}}{(j-\ell)!s^\ell}. $$
The proof follows.
\end{proof}

\begin{lemma}
\label{lemma:MellinRangeclose to 1}
	For $j\geq 0$, $K\geq 2$ and $20(\log K)^{-1} \leq \delta_K \leq \tfrac 12$ we have the estimate
$$
\sum_{c\geq 1}\frac {\mu^2(c)}{\varphi(c)} \int_{K^{-\delta_K}}^{ K^{\delta_K}}\frac{(\log v)^j}{c}   h\Big( \frac{4\pi v}{c} \Big) dv =  \mathcal Mh(1)\frac{(\delta_K \log K)^{j+1}}{4\pi (j+1)}+C_{j,h} +O_{\eps}\Big(j! (\delta_K\log K)^{j} K^{\delta_K(-\frac 12+\eps)}\Big),
$$	
where 
\begin{equation}
C_{j,h}:=\frac{(-1)^j}{j+1}\frac{d^{j+1}}{(ds)^{j+1}} \Big(sZ(s)(4\pi)^{s-1} \mathcal Mh(1-s)\Big)\Bigg|_{s=0}, 
\label{equation definition C_j}
\end{equation}
with
$$Z(s)=\zeta(s+1)  \prod_p \Big( 1+\frac{1}{p-1} \Big( \frac 1{p^{s+1}} - \frac 1{p^{2s+1}}\Big) \Big).$$
\end{lemma}

\begin{proof}
	Define
$$ f_{K,j}(c):= \int_{K^{-\delta_K}}^{ K^{\delta_K}}\frac{(\log v)^j}{c}   h\Big( \frac{4\pi v}{c} \Big) dv.$$
By the restriction on the support of $h$, the function $f_{K,j}$ also has compact support on $\mathbb R_{>0}$. We conclude that its Mellin tranform $\varphi_{K,j}(s)$ is entire. Moreover, 
	\begin{align*}
	 \varphi_{K,j}(s):= \int_0^{\infty} x^{s-1}f_{K,j}(x)dx  &= (4\pi)^{s-1} \mathcal Mh(1-s) \int_{K^{-\delta_K }}^{K^{\delta_K }} (\log v)^j v^{s-1}dv
	 \\ &= \varphi^{+}_{K,j}(s)+\varphi^{-}_{K,j}(s),
\end{align*}	 
where
$$\varphi^{\pm}_{K,j}(s)=\pm  (4\pi)^{s-1} \mathcal Mh(1-s) \int_{1}^{K^{\pm\delta_K }} (\log v)^j v^{s-1}dv$$
are also entire. For any $N\geq 1$, applying \cite[Lemma 2.1]{FPS1} yields the crude bound
\begin{equation}
\varphi^{\pm}_{K,j}(s) \ll_{N,j} K |s|^{-N}  \hspace{3cm} (|\Im(s)|\geq 1, |\Re(s)|\leq 1 ).
\label{equation crude bound}
\end{equation}
Next, Mellin Inversion gives the formula
	$$ f_{K,j}(c) = \frac{1}{2\pi i} \int_{(\frac 12)}c^{-s} ( \varphi^{+}_{K,j}(s)+\varphi^{-}_{K,j}(s)) ds.
	$$
Hence, by absolute convergence, 
	  $$ \sum_{c\geq 1}\frac {\mu^2(c)}{\varphi(c)} f_{K,j}(c)= \frac{1}{2\pi i} \int_{(\frac 12)} Z(s)(\varphi^{+}_{K,j}(s)+\varphi^{-}_{K,j}(s)) ds,   $$
	  where
	  $$Z(s):=\sum_{c=1}^{\infty} \frac {\mu^2(c)}{c^s\varphi(c)} =\zeta(s+1)  \prod_p \Big( 1+\frac{1}{p-1} \Big( \frac 1{p^{s+1}} - \frac 1{p^{2s+1}}\Big) \Big).$$
By applying Lemma \ref{lemma bound incomplete gamma}, we obtain the estimate
$$ \varphi^{-}_{K,j}(s)= (4\pi)^{s-1} \mathcal Mh(1-s) \Big(\int_{0}^{1} (\log v)^j v^{s-1}dv +O\Big(j!(\delta_K \log K)^{j} K^{-\delta_K \Re(s)}\Big) \Big)\;\;\;(\Re(s)> 0, |s|\geq \tfrac 1{10}).$$
Moreover,
$$\int_{0}^{1} (\log v)^j v^{s-1}dv =\frac{(-1)^{j}j!}{s^{j+1}}\;\;\;(\Re(s)> 0).$$

From the rapid decay of $\mathcal Mh(1-s)$ on vertical lines (see \cite[Lemma 2.1]{FPS1}), we see that 
$$ \frac{1}{2\pi i} \int_{(\frac 12)} Z(s)\varphi^{-}_{K,j}(s) ds=\frac{(-1)^jj!}{2\pi i} \int_{(\frac 12)} Z(s)(4\pi)^{s-1} \mathcal Mh(1-s) \frac{ds}{s^{j+1}} +O\Big(j!(\delta_K \log K)^{j} K^{-\frac{\delta_K} 2}\Big).$$
	As for the other part of the integral, by applying \eqref{equation crude bound} we can shift the countour to the left until the line $\Re(s)=-\frac 12+\frac{\eps}2$, and reach the identity
$$ \frac{1}{2\pi i} \int_{(\frac 12)} Z(s)\varphi^{+}_{K,j}(s) ds =(4\pi)^{-1} \mathcal Mh(1) \int_{1}^{K^{\delta_K }} (\log v)^j v^{-1}dv+ \frac 1{2\pi i} \int_{(-\frac 12+\frac{\eps}2)} Z(s)\varphi^+_{K,j}(s) ds. $$	
	In a similar fashion as before, we see that
	\begin{equation*}
 \varphi^{+}_{K,j}(s)= (4\pi)^{s-1} \mathcal Mh(1-s) \Big(\frac{(-1)^{j+1}j!}{s^{j+1}} +O\Big(j!(\delta_K \log K)^{j} K^{\delta_K \Re(s)}\Big)\Big) \;\;\;(\Re(s)<0, |s| \geq  \tfrac 1{10})
\end{equation*}
and deduce that
	\begin{multline*}
	\frac{1}{2\pi i} \int_{(\frac 12)} Z(s)\varphi^{+}_{K,j}(s) ds= \mathcal Mh(1)\frac{(\delta_K \log K)^{j+1}}{4\pi (j+1)} 
	+\frac{(-1)^{j+1}j!}{2\pi i} \int_{(-\frac 12+\frac{\eps}2)} Z(s)(4\pi)^{s-1} \mathcal Mh(1-s) \frac{ds}{s^{j+1}}\\+O_{\eps}\Big(	j! (\delta_K\log K)^{j} K^{(-\frac 12+\eps)\delta_K}\Big).  
\end{multline*}	 
Putting these estimates together, we conclude that 
\begin{multline*}
\sum_{c\geq 1}\frac {\mu^2(c)}{\varphi(c)} f_{K,j}(c) = \mathcal Mh(1)\frac{(\delta_K \log K)^{j+1}}{4\pi (j+1)} +\frac{(-1)^jj!}{2\pi i}\Big( \int_{(\frac 12)}- \int_{(-\frac 12+\frac \eps 2)} \Big) Z(s)(4\pi)^{s-1} \mathcal Mh(1-s) \frac{ds}{s^{j+1}}
\\+O_{\eps}\Big(	j! (\delta_K\log K)^{j} K^{(-\frac 12+\eps)\delta_K}\Big).
\end{multline*}   
	The result follows.
\end{proof}

\begin{lemma}[Range $1-\delta_K<u<1+\delta_K$]
Let $\phi$ be an even Schwartz test function.	We have for $K\geq 2$ and odd $J\geq 1$ that	$$ I_{1-\delta_K,1+\delta_K} =\int_{1}^{1+\delta_K} \widehat{\phi}+\frac{\pi K}{H^{\pm }(K)}\sum_{0\leq j \leq J} \frac{\widehat \phi^{(j)}(1)C_{j,h}}{j!(\log K)^{j+1}} +O_{\eps,J}\Big( \delta_K^{J+2}+\frac 1{(\log K)^{J+2}} +\frac {K^{\delta_K (-\frac 12+\eps)}}{\log K} \Big),$$
	where the constants $C_{j,h}$ are defined in \eqref{equation definition C_j}.
	\label{lemma taylor series}
\end{lemma}

\begin{proof}
By definition of $I_{1-\delta_K,1+\delta_K}$, we need to evaluate the sum
\begin{align*}
	\sum_{c\geq 1}\frac {\mu^2(c)}{\varphi(c)} \int_{1-\delta_K}^{1+\delta_K} \frac { K^u}c  \widehat \phi(u) h\Big( \frac{4\pi  K^{u-1}}{c}\Big) du  
= \frac K{\log K} \sum_{c\geq 1}\frac {\mu^2(c)}{\varphi(c)} \int_{K^{-\delta_K}}^{ K^{\delta_K}}\frac{1}{c} \widehat \phi\Big( \frac{\log v}{\log K}+1 \Big)   h\Big( \frac{4\pi v}{c} \Big) dv.
 \end{align*} 
Taking Taylor series and applying\footnote{The error term obtained after the Taylor series expansion contains the expression $\sum_{c\geq 1} \frac{\mu^2(c)}{\varphi(c)}\int_{K^{-\delta_K}}^{K^{\delta_K}} |\log v|^{J+1} h(\frac{4\pi v}c)dv$, which can be evaluated using Lemma \ref{lemma:MellinRangeclose to 1} whenever $J+1$ is even.} Lemma \ref{lemma:MellinRangeclose to 1}, this is
  \begin{multline*}
	= K\sum_{0\leq j\leq J} \frac{\widehat \phi^{(j)}(1)}{j!(\log K)^{j+1}} \sum_{c\geq 1}\frac {\mu^2(c)}{\varphi(c)} \int_{K^{-\delta_K}}^{ K^{\delta_K}}\frac{(\log v)^j}{c}   h\Big( \frac{4\pi v}{c} \Big) dv\\+O_{\eps,J}\Big( K\delta_K^{J+2} +\frac K{(\log K)^{J+2}}+\frac {K^{1+\delta_K (-\frac 12+\eps)}}{\log K} \Big).
\end{multline*}
Applying Lemma \ref{lemma:MellinRangeclose to 1} once more, we reach the expression 
\begin{align*}
\frac{K \mathcal Mh(1)}{4\pi }\sum_{0\leq j\leq J} \frac{\widehat \phi^{(j)}(1)\delta_K^{j+1}}{(j+1)!}+K\sum_{0\leq j \leq J} \frac{\widehat \phi^{(j)}(1)C_{j,h}}{j!(\log K)^{j+1}}  + O_{\eps,J}\Big( K\delta_K^{J+2} +\frac K{(\log K)^{J+2}}+\frac {K^{1+\delta_K (-\frac 12+\eps)}}{\log K} \Big).
\end{align*}
Applying Lemma \ref{lemma:average h log}, the proof follows.
\end{proof}

Collecting the estimates in this section, we reach the following theorem.

\begin{theorem}
\label{theorem final result}
Let $\phi$ be an even Schwartz test function for which supp$(\widehat \phi) \subset (-2,2)$. Assuming the Riemann Hypothesis for Dirichlet $L$-functions, for $K\geq 2$ we have the estimate 
\begin{multline*}
	\DplusminusAvg = \widehat \phi(0) \Bigg(1 +\frac{ \frac{\int_0^{\infty} h\cdot \log }{\int_0^{\infty} h}-\log (4\pi)}{\log K} \Bigg)+\frac{\phi(0)}2  + \sum_{1 \leq j\leq J} \frac{c_j\widehat \phi^{(j-1)}(0)}{2^{j}(\log K)^j}\\ \mp \int_{1}^{\infty} \widehat{\phi}\pm\sum_{1\leq j \leq J} \frac{D_{j,h}\widehat \phi^{(j-1)}(1)}{(\log K)^{j}} +O_{\eps,J}\Big( \frac 1{(\log K)^{J+1}}\Big),
	  \end{multline*}
where the $c_j$ are defined in Lemma \ref{lemma:Seven}, and
$$ D_{j,h} =-\frac{4\pi}{ \int_{\mathbb R} h \cdot (j-1)!}C_{j-1,h} =\frac{4\pi (-1)^{j}}{j!\int_{\mathbb R^+}h} \frac{d^{j}}{(ds)^{j}} \Big(sZ(s)(4\pi)^{s-1} \mathcal Mh(1-s)\Big)\Bigg|_{s=0}, $$
with
\begin{equation}
\label{equation Z(s)}
Z(s)=\zeta(s+1)  \prod_p \Big( 1+\frac{1}{p-1} \Big( \frac 1{p^{s+1}} - \frac 1{p^{2s+1}}\Big) \Big).
\end{equation}
\end{theorem}

\begin{proof}
Recall \eqref{eq:power saving in one level density}, which is valid for $\sigma=$supp$($sup$(\widehat \phi))<2$:
\begin{equation*}
	\DplusminusAvg =  \widehat \phi(0) \bigg(1 +\frac{\frac{ \int_{\mathbb R^+} h \cdot \log}{\int_{\mathbb R^+} h }-\log (4\pi)}{\log K} \bigg)+2\sum_{p} \frac 1p \widehat \phi \Big( \frac{2 \log p}{\log (K^2)} \Big) \frac{\log p}{\log (K^2)} \mp I_{0,\infty} + O_{\eps}(K^{\frac{\sigma}2-1+\eps}+K^{-\frac 13+\eps}).	 
	 \end{equation*}
We can clearly assume without loss of generality that $J$ is odd.
The sum over primes is estimated in Lemma \ref{lemma:Seven}. Moreover, we recall that for $K$ large enough $I_{0,1-\delta_K}=0$ and therefore we have that 
$$I_{0,\infty}=I_{1-\delta_K,1+\delta_K}+ I_{1+\delta_K,\infty},  $$ 
which together with Lemmas \ref{lemma range u large} and \ref{lemma taylor series} and the choice $\delta_K=3(J+3)\log\log K/\log K$ implies the desired result.	 
\end{proof}

\begin{proof}[Proof of Theorem \ref{theorem main}]
It follows immediately from Theorem \ref{theorem final result} with 

\begin{equation}
S_{j,h}=D_{j,h}=\frac{4\pi (-1)^{j}}{j!\int_{\mathbb R^+}h} \frac{d^{j}}{(ds)^{j}} \Big(sZ(s)(4\pi)^{s-1} \mathcal Mh(1-s)\Big)\Bigg|_{s=0}
\label{equation definition S}
\end{equation}  
(see \eqref{equation Z(s)});
\begin{equation}
 R_{1,h}= \frac{ \int_{\mathbb R^+} h \cdot \log}{\int_{\mathbb R^+} h }-\log (4\pi)+\int_1^{\infty} \frac{\theta(t)-t}{t^2}dt+1;
 \label{equation definition R 1}
\end{equation}
and for $j\geq 2$, 
\begin{equation}
 R_{j,h}= \frac{1 }{(j-2)!}  \int_1^{\infty}(\log t)^{j-2} \Big( \frac{\log t}{j-1}- 1 \Big) \frac{\theta(t)-t}{t^2}dt
 \label{equation definition R 2}
\end{equation}
 (note that these do not depend on $h$).
\end{proof}

\label{section average over k}

\end{document}